\documentclass[a4paper,10pt]{amsart}
\usepackage{pb-diagram}
\usepackage{xypic}
\usepackage{amscd}
\usepackage{bm}
\usepackage{amsmath,amssymb,amscd,bbm,amsthm,mathrsfs,hyperref}
\usepackage{latexsym}
\usepackage{amsfonts}
\usepackage{amsthm}
\usepackage{indentfirst}
 \newtheorem{thm}{Theorem}[section]
 \newtheorem{cor}[thm]{Corollary}
 
 \newtheorem{prop}[thm]{Proposition}
 \newtheorem{defn}[thm]{Definition}
 
 \newtheorem{rem}[thm]{Remark}

 \def\k{\mathbbm{k}}

 \newcommand{\Hom}{\mathrm{Hom}}

\title{DG polynomial algebras and their homological properties}
\author{X.-F. Mao}
\address{Department of Mathematics, Shanghai University, Shanghai 200444, China}
\email{xuefengmao@shu.edu.cn}
\author{X.-D. Gao}
\address{Department of Mathematics, College of Mathematics and Statistics, Kashgar University, Kashgar, Xinjiang, 844006, China}
\email{986242791@qq.com}
\author{Y.-N. Yang}
\address{Department of Mathematics, College of Mathematics and Statistics, Kashgar University, Kashgar, Xinjiang, 844006, China}
\email{1252170187@qq.com}
\author{J.-H. chen}
\address{Department of Mathematics, College of Mathematics and Statistics, Kashgar University, Kashgar, Xinjiang, 844006, China}
\email{sharemaster@sina.com}
\date{}

\subjclass[2010]{Primary 16E45, 16E65, 16W20,16W50}
\keywords{DG polynomial algebra, cohomology graded algebra,  homologically smooth, Gorenstein, Calabi-Yau}
\begin{document}

\maketitle \def\abstactname{abstact}
\begin{abstract}
In this paper, we introduce and study differential graded (DG for short) polynomial algebras. In brief, a DG polynomial algebra $\mathcal{A}$ is a connected cochain DG algebra
such that its underlying graded algebra $\mathcal{A}^{\#}$ is a polynomial algebra $\k[x_1,x_2,\cdots, x_n]$ with $|x_i|=1$, for any $i\in \{1,2,\cdots, n\}$.

We describe all possible differential structures on DG polynomial algebras;
   compute their DG automorphism groups; study their isomorphism problems; and show that they are all homologically smooth and Gorestein DG algebras. Furthermore, it is proved that the DG polynomial algebra $\mathcal{A}$ is a Calabi-Yau DG algebra when its differential $\partial_{\mathcal{A}}\neq 0$ and the trivial  DG polynomial algebra $(\mathcal{A}, 0)$ is Calabi-Yau if and only if $n$ is an odd integer.
\end{abstract}

\maketitle

\section{introduction}

In the literature, there are many papers on the research of homological properties of connected cochain DG algebras. For example, Gorenstein properties of DG algebras are studied in \cite{FHT1,FIJ, FJ1,FJ2,FM,Gam,Jor,Mao,MW2,Sch};
  He-Wu (cf. \cite{HW}) introduce and study Koszul  connected cochain DG algebras; and recently,  the first  author and J.-W. He (cf. \cite{HM}) give a criterion for a connected cochain DG algebra to be $0$-Calabi-Yau, and prove that a locally finite connected cochain DG algebra is $0$-Calabi-Yau if and only if it is defined by a potential. In spite of these, it is still difficult to judge whether a given DG algebra has some good homological properties such as formality, homologically smoothness, Gorensteinness  and
   Calabi-Yau property.

   Generally, the homological properties of a DG algebra are determined by the joint effects of its underlying graded algebra structure and differential structure. However, it is feasible,  at least in some special
cases, to judge some homological properties of a DG algebra from its underlying graded algebra. For example, it is shown in \cite{Mao} that
 a connected cochain DG algebra $\mathcal{B}$ is Gorenstein if its underlying graded algebra $\mathcal{B}^{\#}$ is an Artin-Schelter regular algebra of dimension $2$. Especially, if $\mathcal{B}^{\#}$
is generated by degree $1$ elements $x,y$ and subject to the
relation $xy+yx=0$, then $\mathcal{B}$ is a Koszul Calabi-Yau DG algebra (see \cite{MH}).  Recently, DG down-up algebras are introduced and studied in \cite{MHLX}. It is proved that all Non-trivial Noetherian DG down-up algebras are Calabi-Yau DG algebras.

This paper deals with DG polynomial algebras, which are connected cochain DG algebras
whose underlying graded algebras are polynomial algebras generated by degree $1$ elements.
We describe all possible differential structures on such DG polynomial algebras by the following theorem (see Theorem \ref{diffstructure}).
\\
\begin{bfseries}
Theorem \ A.
\end{bfseries}
Let $(\mathcal{A},\partial_{\mathcal{A}})$ be a connected cochain DG algebra such that $\mathcal{A}^{\#}$ is a polynomial graded algebra $\k[x_1,x_2,\cdots,x_n]$ with $|x_i|=1$, for any $i\in \{1,2,\cdots,n\}$.
Then there exist some $t_1,t_2,\cdots,t_n\in \k$ such that $\partial_{\mathcal{A}}$ is defined  by
$$\partial_{\mathcal{A}}(x_i)=\sum\limits_{j=1}^nt_jx_ix_j=\sum\limits_{j=1}^{i-1}t_jx_jx_i+t_ix_i^2+\sum\limits_{j=i+1}^nt_jx_ix_j,$$ for any $i\in \{1,2,\cdots,n\}$.
 Conversely, for any point $(t_1,t_2,\cdots, t_n)$ in the affine $n$-space $ \Bbb{A}_{\k}^n$, we can define a differential $\partial$ on $\k[x_1,x_2,\cdots,x_n]$ by
$$\partial(x_i)=\sum\limits_{j=1}^{i-1}t_jx_jx_i+t_ix_i^2+\sum\limits_{j=i+1}^nt_jx_ix_j,   \forall i\in \{1,2,\cdots,n\},$$
such that $(\k[x_1,x_2,\cdots,x_n],\partial)$ is a DG polynomial algebra.

By Theorem A, we can define $\mathcal{A}(t_1,t_2,\cdots, t_n)$ as a cochain DG algebra such that
$\mathcal{A}(t_1,t_2,\cdots, t_n)^{\#}=\k[x_1,x_2,\cdots,x_n]$ and its differential $\partial_{\mathcal{A}}$ is defined by $$\partial_{\mathcal{A}}(x_i)=\sum\limits_{j=1}^{i-1}t_jx_jx_i+t_ix_i^2+\sum\limits_{j=i+1}^nt_jx_ix_j, \forall i\in \{1,2,\cdots,n\}.$$
Then the set
$$\Omega(x_1,x_2,\cdots,x_n)=\{\mathcal{A}(t_1,t_2,\cdots,t_n)|t_i\in
\k, i=1,2,\cdots, n\}\cong \Bbb{A}_{\k}^n.$$

To consider the homological properties of $\mathcal{A}(t_1,t_2,\cdots, t_n)$, it is necessary to study the isomorphism problem of DG polynomial algebras.  We have the following theorem (see Theorem \ref{isom}).
\\
\begin{bfseries}
Theorem \ B.
\end{bfseries}
Let $\mathcal{A}(t_1,t_2,\cdots,t_n)$ and
$\mathcal{A}(t_1',t_2',\cdots,t_n')$ be two points in the
space $\Omega(x_1,x_2,\cdots,x_n)$. Then the DG algebras
$$\mathcal{A}(t_1,t_2,\cdots,t_n)\cong
\mathcal{A}(t_1',t_2',\cdots,t_n')$$ if and only if there
is a matrix $M\in \mathrm{GL}_n(\k)$ such that
$$(t_1',t_2',\cdots,t_n')=(t_1,t_2,\cdots,t_n)M.$$

By Theorem B,  we have only two isomorphism classes in $\Omega(x_1,x_2,\cdots,x_n)$ represented by $\mathcal{A}(0,0,\cdots, 0)$ and $\mathcal{A}(1,0,\cdots,0)$ (see Corollary \ref{isomor}), and we obtain the automorphism group (see Corollary \ref{auto})
$$\mathrm{Aut}_{dg}(\mathcal{A}(t_1,t_2,\cdots,t_n))\cong \{M\in \mathrm{GL}_n(\k)| (t_1,t_2,\cdots,t_n)=(t_1,t_2,\cdots, t_n)M   \}.$$
Generally, it is difficult to determine whether  a given DG algebra has some nice homological properties.
 Comparatively speaking, it is much easier to compute its cohomology graded algebra. For this, we compute the cohomology graded algebra of $\mathcal{A}(1,0,\cdots, 0)$, which is (see Proposition \ref{cohomology})
$$\k[\lceil x_2^2\rceil,\lceil x_2x_3\rceil, \cdots,\lceil x_2x_n\rceil, \lceil x_3^2\rceil, \cdots, \lceil x_3x_n\rceil, \cdots,  \lceil x_{n-1}^2\rceil, \lceil x_{n-1}x_n\rceil, \lceil x_n^2\rceil ].$$
It implies that any DG polynomial algebra $\mathcal{A}(t_1,t_2,\cdots,t_n)$ is a formal, homologically smooth and Gorenstein DG algebra (see Theorem \ref{asreg}).

It is natural for one to ask whether DG polynomial algebras are Calabi-Yau. In \cite{MH}, it is proved that a connected cochain DG algebra $\mathcal{A}$ is a Kozul Calabi-Yau DG algebra if $H(\mathcal{A})$  belongs to one of the following cases:
\begin{align*}
& (a) H(\mathcal{A})\cong \k;  \quad \quad (b) H(\mathcal{A})= \k[\lceil z\rceil], z\in \mathrm{ker}(\partial_{\mathcal{A}}^1); \\
& (c) H(\mathcal{A})= \frac{\k\langle \lceil z_1\rceil, \lceil z_2\rceil\rangle}{(\lceil z_1\rceil\lceil z_2\rceil +\lceil z_2\rceil \lceil z_1\rceil)}, z_1,z_2\in \mathrm{ker}(\partial_{\mathcal{A}}^1).
\end{align*}
Recently, it is proved in \cite[Proposition 6.5]{MHLX} that a connected cochain DG algebra $\mathcal{A}$ is Calabi-Yau if $H(\mathcal{A})=\k[\lceil z_1\rceil, \lceil z_2\rceil]$ where $z_1\in \mathrm{ker}(\partial_{\mathcal{A}}^1)$ and $z_2\in \mathrm{ker}(\partial_{\mathcal{A}}^2)$. And a connected cochain DG algebra $\mathcal{A}$ is not Calabi-Yau if $H(\mathcal{A})=\k[\lceil z_1\rceil, \lceil z_2\rceil]$ where $z_1\in \mathrm{ker}(\partial_{\mathcal{A}}^1)$ and $z_2\in \mathrm{ker}(\partial_{\mathcal{A}}^1)$ (see \cite[Theorem B]{MH}). These motivate us to consider more general cases.

We have the following two theorems (see Theorem \ref{cyprop} and Theorem \ref{degreeone}).
 \\
\begin{bfseries}
Theorem \ C.
\end{bfseries}
 Let $\mathcal{A}$ be a connected cochain DG algebra such that $$H(\mathcal{A})=\k[\lceil y_1\rceil, \cdots, \lceil y_m\rceil ],$$ for some central,  cocycle and degree $2$ elements $y_1,\cdots, y_m$  in $\mathcal{A}$. Then $\mathcal{A}$ is a $-m$-Calabi-Yau DG algebra.
 \\
 \begin{bfseries}
Theorem \ D.
\end{bfseries}
Let $\mathcal{A}$ be a connected cochain DG algebra such that $$H(\mathcal{A})=\k[\lceil y_1\rceil, \cdots, \lceil y_m\rceil ],$$ for some central, cocycle and degree $1$ elements $y_1,\cdots, y_m$  in $\mathcal{A}$. Then $\mathcal{A}$ is a Koszul, homologically smooth and Gorenstein DG algebra. Moreover, $\mathcal{A}$ is Calabi-Yau if and only if $n$ is an odd integer.

With the help of Theorem C and Theorem D, we get the following conclusion (see Corollary \ref{non-zero} and Corollary \ref{zerodiff}): $\mathcal{A}(t_1,t_2,\cdots, t_n)$ is a Calabi-Yau DG algebra if $(t_1,t_2,\cdots, t_n)\neq (0,0,\cdots, 0)$ and $\mathcal{A}(0,0,\cdots, 0)$ is a Calabi-Yau DG algebra if and only if $n$ is an odd integer.

----------------------------------------------------------------------
\section{Notations and conventions}
 We assume that the reader is familiar with basic definitions concerning DG homological
algebra.  If this is not the case, we refer to \cite{AFH, FHT2, MW1,MW2,FJ2} for more details on them. We begin by fixing some notations and terminology.
There are some overlaps here in \cite{MHLX}.

Throughout this paper, $\k$ is an algebraically closed field of characteristic $0$.
For any $\k$-vector space $V$, we write $V^*=\Hom_{\k}(V,\k)$. Let $\{e_i|i\in I\}$ be a basis of a finite dimensional $\k$-vector space $V$.  We denote the dual basis of $V$ by $\{e_i^*|i\in I\}$, i.e., $\{e_i^*|i\in I\}$ is a basis of $V^*$ such that $e_i^*(e_j)=\delta_{i,j}$. For any graded vector space $W$ and $j\in\Bbb{Z}$,  the $j$-th suspension $\Sigma^j W$ of $W$ is a graded vector space defined by $(\Sigma^j W)^i=W^{i+j}$.

A cochain DG algebra is
a graded
$\k$-algebra $\mathcal{A}$ together with a differential $\partial_{\mathcal{A}}: \mathcal{A}\to \mathcal{A}$  of
degree $1$ such that
\begin{align*}
\partial_{\mathcal{A}}(ab) = (\partial_{\mathcal{A}} a)b + (-1)^{|a|}a(\partial_{\mathcal{A}} b)
\end{align*}
for all graded elements $a, b\in \mathcal{A}$.
For any DG algebra $\mathcal{A}$,  we denote $\mathcal{A}\!^{op}$ as its opposite DG
algebra, whose multiplication is defined as
 $a \cdot b = (-1)^{|a|\cdot|b|}ba$ for all
graded elements $a$ and $b$ in $\mathcal{A}$.  A cochain DG algebra $\mathcal{A}$ is called
non-trivial if $\partial_{\mathcal{A}}\neq 0$, and $\mathcal{A}$ is said to be connected if its underlying graded algebra $\mathcal{A}^{\#}$ is a connected  graded algebra. Given a cochain DG algebra $\mathcal{A}$, we denote by $\mathcal{A}^i$ its $i$-th homogeneous component.  The differential $\partial_{\mathcal{A}}$ is a sequence of linear maps $\partial_{\mathcal{A}}^i: \mathcal{A}^i\to \mathcal{A}^{i+1}$ such that $\partial_{\mathcal{A}}^{i+1}\circ \partial_{\mathcal{A}}^i=0$, for all $i\in \Bbb{Z}$. The cohomology graded algebra of $\mathcal{A}$ is the graded algebra $$H(\mathcal{A})=\bigoplus_{i\in \Bbb{Z}}\frac{\mathrm{ker}(\partial_{\mathcal{A}}^i)}{\mathrm{im}(\partial_{\mathcal{A}}^{i-1})}.$$
 For any cocycle element $z\in \mathrm{ker}(\partial_{\mathcal{A}}^i)$, we write $\lceil z \rceil$ as the cohomology class in $H(\mathcal{A})$ represented by $z$. One sees that $H(\mathcal{A})$ is a connected graded algebra if $\mathcal{A}$ is a connected cochain DG algebra.
  For any connected cochain DG algebra $\mathcal{A}$, we denote by $\frak{m}_{\mathcal{A}}$ its maximal DG
ideal $$ \cdots\to 0\to \mathcal{A}^1\stackrel{\partial^1_{\mathcal{A}}}{\to}
\mathcal{A}^2\stackrel{\partial^2_{\mathcal{A}}}{\to} \cdots \stackrel{\partial^{n-1}_{\mathcal{A}}}{\to}
\mathcal{A}^n\stackrel{\partial^n_{\mathcal{A}}}{\to}
 \cdots .$$
 Clearly,  $\k$ has a structure of DG $\mathcal{A}$-module via the augmentation map $$\varepsilon: \mathcal{A}\to \mathcal{A}/\frak{m}_{\mathcal{A}}=\k.$$
It is easy to check that the enveloping DG algebra $\mathcal{A}^e = \mathcal{A}\otimes \mathcal{A}\!^{op}$ of $\mathcal{A}$
is also a connected cochain DG algebra with $H(\mathcal{A}^e)\cong H(\mathcal{A})^e$, and $$\frak{m}_{\mathcal{A}^e}=\frak{m}_{\mathcal{A}}\otimes \mathcal{A}^{op} + \mathcal{A}\otimes
\frak{m}_{\mathcal{A}^{op}}.$$

A morphism $f:\mathcal{A}\to \mathcal{A}'$ of DG algebras is a chain map of complexes which respects multiplication and unit; $f$ is said to be a DG algebra isomorphism (resp. quasi-isomorphism) if $f$ (resp. $H(f)$) is an isomorphism.  A DG algebra isomorphism $f$ is called a DG automorphism when $\mathcal{A}'=\mathcal{A}$. The set of all DG algebra automorphisms of $\mathcal{A}$ is a group, denoted by $\mathrm{Aut}_{dg}(\mathcal{A})$.

Let $\mathcal{A}$ be a connected cochain DG algebra. The derived category of left DG modules over $\mathcal{A}$ (DG $\mathcal{A}$-modules for short) is denoted by $\mathrm{D}(\mathcal{A})$.  A DG $\mathcal{A}$-module  $M$ is compact if the functor $\Hom_{\mathrm{D}(A)}(M,-)$ preserves
all coproducts in $\mathrm{D}(\mathcal{A})$.
 By \cite[Proposition 3.3]{MW1},
a DG $\mathcal{A}$-module  is compact if and only if it admits a minimal semi-free resolution with a finite semi-basis. The full subcategory of $\mathrm{D}(\mathcal{A})$ consisting of compact DG $\mathcal{A}$-modules is denoted by $\mathrm{D}^c(\mathcal{A})$.

In the rest of this section, we review some important homological properties for DG algebras.
 We have the following definitions.
\begin{defn}\label{basicdef}
{\rm Let $\mathcal{A}$ be a connected cochain DG algebra.
\begin{enumerate}
\item  If $\dim_{\k}H(R\Hom_{\mathcal{A}}(\k,\mathcal{A}))=1$, then $\mathcal{A}$ is called Gorenstein (cf. \cite{FHT1});
\item  If $\mathcal{A}$ can be connected with the trivial DG algebra $(H(\mathcal{A}),0)$ by a zig-zag $\leftarrow \rightarrow\leftarrow \cdots\rightarrow $ of quasi-isomorphisms, then $\mathcal{A}$ is called formal (cf.\cite{Kal,Lunt});
\item  If ${}_{\mathcal{A}}\k$, or equivalently ${}_{\mathcal{A}^e}\mathcal{A}$, has a minimal semi-free resolution with a semi-basis concentrated in degree $0$, then $\mathcal{A}$ is called Koszul (cf. \cite{HW});
\item If ${}_{\mathcal{A}}\k$, or equivalently the DG $\mathcal{A}^e$-module $\mathcal{A}$ is compact, then $\mathcal{A}$ is called homologically smooth (cf. \cite[Corollary 2.7]{MW3});
\item If $\mathcal{A}$ is homologically smooth and $$R\Hom_{\mathcal{A}^e}(\mathcal{A}, \mathcal{A}^e)\cong
\Sigma^{-n}\mathcal{A}$$ in  the derived category $\mathrm{D}((\mathcal{A}^e)^{op})$ of right DG $\mathcal{A}^e$-modules, then $\mathcal{A}$ is called an $n$-Calabi-Yau DG algebra  (cf. \cite{Gin,VdB}).
\end{enumerate}}
 \end{defn}
 The motivation of this paper is to study whether DG polynomial algebras have these homological properties mentioned in Definition \ref{basicdef}.

\section{differential structures on polynomial algebra}
In this section, we study the differential structures of DG polynomial algebras. We have the following theorem.
\begin{thm}\label{diffstructure}
Let $(\mathcal{A},\partial_{\mathcal{A}})$ be a connected cochain DG algebra such that $\mathcal{A}^{\#}$ is a polynomial graded algebra $\k[x_1,x_2,\cdots,x_n]$ with $|x_i|=1$, for any $i\in \{1,2,\cdots,n\}$.
Then there exist some $t_1,t_2,\cdots,t_n\in \k$ such that $\partial_{\mathcal{A}}$ is defined  by
$$\partial_{\mathcal{A}}(x_i)=\sum\limits_{j=1}^nt_jx_ix_j=\sum\limits_{j=1}^{i-1}t_jx_jx_i+t_ix_i^2+\sum\limits_{j=i+1}^nt_jx_ix_j, \forall i\in \{1,2,\cdots,n\}. $$  Conversely, for any point $(t_1,t_2,\cdots, t_n)\in \Bbb{A}_{\k}^n$, we can define a differential $\partial$ on $\k[x_1,x_2,\cdots,x_n]$ by
$$\partial(x_i)=\sum\limits_{j=1}^{i-1}t_jx_jx_i+t_ix_i^2+\sum\limits_{j=i+1}^nt_jx_ix_j,   \forall i\in \{1,2,\cdots,n\},$$
such that $(\k[x_1,x_2,\cdots,x_n],\partial)$ is a DG polynomial algebra.
\end{thm}
\begin{proof}
Since the differential $\partial_{\mathcal{A}}$ of $\mathcal{A}$ is a $\k$-linear map of degree $1$, we may let
$$
\partial_{\mathcal{A}}(x_i) = \sum\limits_{j=1}^n\sum\limits_{k=j}^nc^i_{j,k}x_jx_k
$$
where $c^i_{j,k}\in \k$, for any $i, j\in \{1,2,\cdots,n\}$ and $k\in \{j+1,\cdots, n\}$
 By
definition, $(\mathcal{A},\partial_{\mathcal{A}})$ is a cochain DG
algebra if and only if $\partial_{\mathcal{A}}$ satisfies the
Leibniz rule and
\begin{align}\label{eqs}
\begin{cases}
\partial_{\mathcal{A}}\circ \partial_{\mathcal{A}}(x_i) = 0, \forall i\in\{1,2,\cdots,n\}   \\
\partial_{\mathcal{A}}(x_jx_k-x_kx_j) = 0, \forall 1\le j<k\le n.
\end{cases}
\end{align}
Since \begin{align*} \partial_{\mathcal{A}}(x_jx_k-x_kx_j)&=\partial_{\mathcal{A}}(x_j)x_k-x_j\partial_{\mathcal{A}}(x_k)-\partial_{\mathcal{A}}(x_k)x_j+x_k\partial_{\mathcal{A}}(x_j)\\
&=2\partial_{\mathcal{A}}(x_j)x_k-2x_j\partial_{\mathcal{A}}(x_k)\\
&=2[\sum\limits_{p=1}^n\sum\limits_{q=p}^nc^j_{p,q}x_px_q]x_k-2x_j[\sum\limits_{p=1}^n\sum\limits_{q=p}^nc^k_{p,q}x_px_q]\\
&=2[\sum\limits_{p=1}^n\sum\limits_{q=p}^nc^j_{p,q}x_px_q]x_k-2x_j[\sum\limits_{q=1}^n\sum\limits_{p=1}^qc^k_{p,q}x_px_q],
\end{align*}
\begin{align*}
&\quad\quad\quad [\sum\limits_{p=1}^n\sum\limits_{q=p}^nc^j_{p,q}x_px_q]x_k \\
&=\sum\limits_{p=1}^{j-1}\sum\limits_{q=p}^{j-1}c_{p,q}^jx_px_qx_k +\sum\limits_{p=1}^{j-1}c_{p,j}^jx_px_jx_k +\sum\limits_{p=1}^{j-1}\sum\limits_{q=j+1}^nc_{p,q}^jx_px_qx_k+c_{j,j}^jx_j^2x_k\\
&\quad \quad +\sum\limits_{q=j+1}^nc_{j,q}^jx_jx_qx_k +\sum\limits_{p=j+1}^n\sum\limits_{q=p}^nc_{p,q}^jx_px_qx_k\\
&=R_1+R_2
\end{align*}
and
\begin{align*}
&\quad\quad\quad x_j[\sum\limits_{q=1}^n\sum\limits_{p=1}^qc^k_{p,q}x_px_q] \\
&=\sum\limits_{q=1}^{k-1}\sum\limits_{p=1}^qc_{p,q}^kx_jx_px_q + \sum\limits_{p=1}^{k-1}c_{p,k}^kx_jx_px_k +c_{k,k}^kx_jx_k^2+\sum\limits_{q=k+1}^n\sum\limits_{p=1}^{k-1}c_{p,q}^kx_jx_px_q\\
&\quad\quad +\sum\limits_{q=k+1}^nc_{k,q}^kx_jx_kx_q+\sum\limits_{q=k+1}^n\sum\limits_{p=k+1}^qc_{p,q}^kx_jx_px_q\\
&=S_1+S_2,
\end{align*}
where
\begin{align*}
\begin{cases}
R_1= \sum\limits_{p=1}^{j-1}\sum\limits_{q=p}^{j-1}c_{p,q}^jx_px_qx_k+\sum\limits_{p=1}^{j-1}\sum\limits_{q=j+1}^nc_{p,q}^jx_px_qx_k+\sum\limits_{p=j+1}^n\sum\limits_{q=p}^nc_{p,q}^jx_px_qx_k \\
R_2= \sum\limits_{p=1}^{j-1}c_{p,j}^jx_px_jx_k + c_{j,j}^jx_j^2x_k +\sum\limits_{q=j+1}^nc_{j,q}^jx_jx_qx_k  \\
S_1= \sum\limits_{q=1}^{k-1}\sum\limits_{p=1}^qc_{p,q}^kx_jx_px_q + \sum\limits_{q=k+1}^n\sum\limits_{p=1}^{k-1}c_{p,q}^kx_jx_px_q + \sum\limits_{q=k+1}^n\sum\limits_{p=k+1}^qc_{p,q}^kx_jx_px_q \\
S_2=\sum\limits_{p=1}^{k-1}c_{p,k}^kx_jx_px_k +c_{k,k}^kx_jx_k^2+\sum\limits_{q=k+1}^nc_{k,q}^kx_jx_kx_q,
\end{cases}
\end{align*}
we obtain that $\partial_{\mathcal{A}}(x_jx_k-x_kx_j) = 0$ if and only if $R_1+R_2-S_1-S_2=0$, which is equivalent to
\begin{align*}
\begin{cases}
R_1=0\\
S_1=0\\
R_2-S_2=0.
\end{cases}
\end{align*}
Since
$$R_1=[\sum\limits_{p=1}^{j-1}\sum\limits_{q=p}^{j-1}c_{p,q}^jx_px_q+\sum\limits_{p=1}^{j-1}\sum\limits_{q=j+1}^nc_{p,q}^jx_px_q+\sum\limits_{p=j+1}^n\sum\limits_{q=p}^nc_{p,q}^jx_px_q]x_k
$$
and $$S_1= x_j[\sum\limits_{q=1}^{k-1}\sum\limits_{p=1}^qc_{p,q}^kx_px_q + \sum\limits_{q=k+1}^n\sum\limits_{p=1}^{k-1}c_{p,q}^kx_px_q + \sum\limits_{q=k+1}^n\sum\limits_{p=k+1}^qc_{p,q}^kx_px_q],$$
we get that $R_1=S_1=0$ if and only if
\begin{align*}
\begin{cases}
 c_{p,q}^j=0, \forall p\in \{1,\cdots, j-1\}, q\in \{p,\cdots,j-1\}\\
c_{p,q}^j=0, \forall p\in \{1,\cdots,j-1\}, q\in \{j+1,\cdots,n\} \\
c_{p,q}^j=0, \forall p\in \{j+1,\cdots, n\}, q\in\{p,\cdots,n\}
\end{cases}
\end{align*}
and
\begin{align*}
\begin{cases}
c_{p,q}^k=0, \forall q\in \{1,\cdots, k-1\}, p\in \{1,\cdots, q\}\\
c_{p,q}^k=0, \forall q\in \{k+1,\cdots, n\}, p\in \{1,\cdots, k-1\}\\
c_{p,q}^k=0, \forall q\in \{k+1,\cdots, n\}, p\in \{k+1,\cdots, q\},
\end{cases}
\end{align*}
which are equivalent to
\begin{align*}
\begin{cases}
c_{p,q}^j=0, \text{if}\quad p\neq j \quad \text{and}\quad q\neq j\\
c_{p,q}^k=0, \text{if}\quad p\neq k \quad \text{and}\quad q\neq k.
\end{cases}
\end{align*}
Since \begin{align*}
R_2-S_2&=\sum\limits_{p=1}^{j-1}c_{p,j}^jx_px_jx_k + c_{j,j}^jx_j^2x_k +\sum\limits_{q=j+1}^nc_{j,q}^jx_jx_qx_k \\
 & \quad\quad- [\sum\limits_{p=1}^{k-1}c_{p,k}^kx_jx_px_k +c_{k,k}^kx_jx_k^2+\sum\limits_{q=k+1}^nc_{k,q}^kx_jx_kx_q]\\
&=\sum\limits_{p=1}^{j-1}(c_{p,j}^j-c_{p,k}^k)x_px_jx_k +c_{j,j}^jx_j^2x_k + \sum\limits_{q=k+1}^n(c_{j,q}^j-c_{k,q}^k)x_jx_kx_q  \\
& \quad\quad+\sum\limits_{q=j+1}^{k-1}c_{j,q}^jx_jx_qx_k + c_{j,k}^jx_jx_k^2- c_{j,k}^kx_j^2x_k -\sum\limits_{p=j+1}^{k-1}c_{p,k}^kx_jx_px_k-c_{k,k}^kx_jx_k^2 \\
&=\sum\limits_{p=1}^{j-1}(c_{p,j}^j-c_{p,k}^k)x_px_jx_k +(c_{j,j}^j-c_{j,k}^k)x_j^2x_k+\sum\limits_{q=k+1}^n(c_{j,q}^j-c_{k,q}^k)x_jx_kx_q\\
& \quad\quad+ \sum\limits_{r=j+1}^{k-1}(c_{j,r}^j- c_{r,k}^k )x_jx_rx_k +(c_{j,k}^j -c_{k,k}^k )x_jx_k^2,
\end{align*}
we conclude that $R_2-S_2=0$ if and only if
$$
\begin{cases}
c_{p,j}^j=c_{p,k}^k, \forall p\in \{1,2,\cdots, j-1\}\\
c_{j,j}^j=c_{j,k}^k \\
c_{j,q}^j=c_{k,q}^k, \forall q\in \{k+1,k+2,\cdots, n\} \\
c_{j,r}^j=c_{r,k}^k, \forall r\in \{ j+1,j+2,\cdots, k-1\}\\
c_{j,k}^j=c_{k,k}^k,
\end{cases}
$$
which is equivalent to
\begin{align*}
\begin{cases}
c_{p,j}^j=c_{p,k}^k, \forall p\in \{1,2,\cdots, j\}\\
c_{j,r}^j=c_{r,k}^k, \forall r\in \{ j+1,j+2,\cdots, k-1\}\\
c_{j,q}^j=c_{k,q}^k, \forall q\in \{k,k+2,\cdots, n\}.
\end{cases}
\end{align*}
Therefore, for any $1\le j<k\le n$, $\partial_{\mathcal{A}}(x_jx_k-x_kx_j) = 0$ if and only if
\begin{align}\label{fincond}
\begin{cases}
c_{p,q}^j=0, \text{if}\quad p\neq j \quad \text{and}\quad q\neq j\\
c_{p,q}^k=0, \text{if}\quad p\neq k \quad \text{and}\quad q\neq k\\
c_{p,j}^j=c_{p,k}^k, \forall p\in \{1,2,\cdots, j\}\\
c_{j,r}^j=c_{r,k}^k, \forall r\in \{ j+1,j+2,\cdots, k-1\}\\
c_{j,q}^j=c_{k,q}^k, \forall q\in \{k,k+2,\cdots, n\}.
\end{cases}
\end{align}\label{diff}
Let $t_i=c_{1,i}^1, i=1,2,\cdots, n$. Then (\ref{fincond}) implies \begin{align}
\partial_{\mathcal{A}}(x_i)=\sum\limits_{j=1}^{i-1}t_jx_jx_i+t_ix_i^2+\sum\limits_{j=i+1}^nt_jx_ix_j, \forall i\in \{1,2,\cdots,n\}.
\end{align}
For any $(t_1,t_2,\cdots, t_n)\in \Bbb{A}_{\k}^n$, we claim that $\partial_{\mathcal{A}}\circ\partial_{\mathcal{A}}(x_i)=0$, if (\ref{diff}) holds,
$\forall i\in \{1,2,\cdots, n\}$. Indeed, if (\ref{diff}) holds, then
\begin{align*}
&\quad\quad\quad \partial_{\mathcal{A}}\circ\partial_{\mathcal{A}}(x_i)=\partial_{\mathcal{A}}[\sum\limits_{j=1}^{i-1}t_jx_jx_i+t_ix_i^2+\sum\limits_{j=i+1}^nt_jx_ix_j]\\
&=\sum\limits_{j=1}^{i-1}t_j\partial_{\mathcal{A}}(x_j)x_i-\sum\limits_{j=1}^{i-1}t_jx_j\partial_{\mathcal{A}}(x_i)+\sum\limits_{j=i+1}^nt_j\partial_{\mathcal{A}}(x_i)x_j-\sum\limits_{j=i+1}^nt_jx_i\partial_{\mathcal{A}}(x_j)\\
&=\sum\limits_{j=1}^{i-1}t_j[(\sum\limits_{l=1}^{j-1}t_lx_lx_j+t_jx_j^2+\sum\limits_{l=j+1}^nt_lx_jx_l)x_i-x_j(\sum\limits_{l=1}^{i-1}t_lx_lx_i+t_ix_i^2+ \sum\limits_{l=i+1}^nt_lx_ix_l)]\\
& + \sum\limits_{j=i+1}^nt_j[(\sum\limits_{l=1}^{i-1}t_lx_lx_i+t_ix_i^2+ \sum\limits_{l=i+1}^nt_lx_ix_l)x_j-x_i(\sum\limits_{l=1}^{j-1}t_lx_lx_j+t_jx_j^2+\sum\limits_{l=j+1}^nt_lx_jx_l)]\\
&=\sum\limits_{j=1}^{i-1}t_jx_i[\sum\limits_{l=1}^{j-1}t_lx_lx_j+t_jx_j^2+\sum\limits_{l=j+1}^nt_lx_jx_l-x_j(\sum\limits_{l=1}^{i-1}t_lx_l+t_ix_i+ \sum\limits_{l=i+1}^nt_lx_l)]\\
& + \sum\limits_{j=i+1}^nt_jx_i[(\sum\limits_{l=1}^{i-1}t_lx_l+t_ix_i+ \sum\limits_{l=i+1}^nt_lx_l)x_j-(\sum\limits_{l=1}^{j-1}t_lx_lx_j+t_jx_j^2+\sum\limits_{l=j+1}^nt_lx_jx_l)].
\end{align*}
Since
\begin{align*}
&\sum\limits_{j=1}^{i-1}t_j[\sum\limits_{l=1}^{j-1}t_lx_lx_j+t_jx_j^2+\sum\limits_{l=j+1}^nt_lx_jx_l-x_j(\sum\limits_{l=1}^{i-1}t_lx_l+t_ix_i+ \sum\limits_{l=i+1}^nt_lx_l)]\\
& + \sum\limits_{j=i+1}^nt_j[(\sum\limits_{l=1}^{i-1}t_lx_l+t_ix_i+ \sum\limits_{l=i+1}^nt_lx_l)x_j-(\sum\limits_{l=1}^{j-1}t_lx_lx_j+t_jx_j^2+\sum\limits_{l=j+1}^nt_lx_jx_l)]\\
&=\sum\limits_{j=1}^{i-1}t_j[-\sum\limits_{l=j}^{i-1}t_lx_lx_j+t_jx_j^2-t_ix_ix_j + \sum\limits_{l=j+1}^it_lx_jx_l]\\
&\quad +\sum\limits_{j=i+1}^nt_j[ -\sum\limits_{l=i}^{j-1}t_lx_lx_j +t_ix_ix_j+\sum\limits_{l=i+1}^jt_lx_lx_j-t_jx_j^2                                                                           ]\\
&=\sum\limits_{j=1}^{i-1}t_j[-\sum\limits_{l=j+1}^{i-1}t_lx_lx_j +\sum\limits_{l=j+1}^{i-1}t_lx_jx_l]+\sum\limits_{j=i+1}^nt_j[ -\sum\limits_{l=i+1}^{j-1}t_lx_lx_j+\sum\limits_{l=i+1}^{j-1}t_lx_lx_j] \\
&=0,
\end{align*}
we have $\partial_{\mathcal{A}}^2(x_i)=0$,  $\forall i\in \{1,2,\cdots,n\}$.

Therefore, $(\mathcal{A},\partial_{\mathcal{A}})$ is a cochain DG algebra if there exist some $t_1,t_2,\cdots,t_n\in \k$ such that
$$\partial_{\mathcal{A}}(x_i)=\sum\limits_{j=1}^{i-1}t_jx_jx_i+t_ix_i^2+\sum\limits_{j=i+1}^nt_jx_ix_j, \forall i\in \{1,2,\cdots,n\}.$$
Conversely, $\forall (t_1,t_2,\cdots, t_n)\in \Bbb{A}_{\k}^n$, we can define a differential $\partial$ on $\k[x_1,x_2,\cdots,x_n]$ by
$$\partial(x_i)=\sum\limits_{j=1}^{i-1}t_jx_jx_i+t_ix_i^2+\sum\limits_{j=i+1}^nt_jx_ix_j,   \forall i\in \{1,2,\cdots,n\},$$
such that $(\k[x_1,x_2,\cdots,x_n],\partial)$ is a DG polynomial algebra, since one can check as above that $$
\begin{cases}
\partial \circ \partial (x_i) = 0, \forall i\in\{1,2,\cdots,n\}   \\
\partial (x_jx_k-x_kx_j) = 0, \forall 1\le j<k\le n,
\end{cases}
$$
if $\partial$ satisfies the
Leibniz rule.
\end{proof}
By Theorem \ref{diffstructure}, the following definition is reasonable.
\begin{defn}{\rm
 Define $\mathcal{A}(t_1,t_2,\cdots, t_n)$ as a cochain DG algebra such that
$$\mathcal{A}(t_1,t_2,\cdots, t_n)^{\#}=\k[x_1,x_2,\cdots,x_n],\quad |x_i|=1,$$ and its differential $\partial_{\mathcal{A}}$ is defined by $$\partial_{\mathcal{A}}(x_i)=\sum\limits_{j=1}^{i-1}t_jx_jx_i+t_ix_i^2+\sum\limits_{j=i+1}^nt_jx_ix_j, \forall i\in \{1,2,\cdots,n\}.$$
Define
$\Omega(x_1,x_2,\cdots,x_n)=\{\mathcal{A}(t_1,t_2,\cdots,t_n)|t_i\in
\k, i=1,2,\cdots, n\}$.
Clearly, $$\Omega(x_1,x_2,\cdots,x_n)\cong \Bbb{A}_{\k}^n.$$}
\end{defn}

\section{isomorphism classes of  dg  polynomial algebras}
In this section, we consider the isomorphism problem for DG polynomial algebras in $\Omega(x_1,x_2,\cdots,x_n)$. We have the following theorem.
\begin{thm}\label{isom}
Let $\mathcal{A}(t_1,t_2,\cdots,t_n)$ and
$\mathcal{A}(t_1',t_2',\cdots,t_n')$ be two points in the
space $\Omega(x_1,x_2,\cdots,x_n)$. Then the DG algebras
$$\mathcal{A}(t_1,t_2,\cdots,t_n)\cong
\mathcal{A}(t_1',t_2',\cdots,t_n')$$ if and only if there
is a matrix $M\in \mathrm{GL}_n(\k)$ such that
$$(t_1',t_2',\cdots,t_n')=(t_1,t_2,\cdots,t_n)M.$$
\end{thm}
\begin{proof}
We write
$\mathcal{A}=\mathcal{A}(t_1,t_2,\cdots,t_n)$ and
$\mathcal{A}'= \mathcal{A}(t_1',t_2',\cdots, t_n')$ for
simplicity. In order to distinguish, we assume that $\mathcal{A}'^{\#}=\k[x_1',x_2',\cdots,x_n']$ with  $|x_i'|=1$, for any $i\in \{1,2,\cdots,n\}$.
If the DG
algebras $\mathcal{A}\cong \mathcal{A}'$, then there exists an
isomorphism $f: \mathcal{A}\to \mathcal{A}'$ of DG algebras. Since
$f^1: \mathcal{A}^1 \to \mathcal{A}'^1$ is a $\k$-linear
isomorphism, we may let $$\left(\begin{array}{c}
                           f(x_1) \\
                           f(x_2) \\
                           \vdots \\
                           f(x_n)
                         \end{array}
                       \right)= M\left(
                         \begin{array}{c}
                           x_1' \\
                           x_2' \\
                           \vdots \\
                           x_n'
                         \end{array}
                       \right)$$
for some  $M=(a_{ij})_{n\times n} \in \mathrm{GL}_n(\k)$.
Since $f$ is a chain map, we have $f\circ \partial_{\mathcal{A}}=\partial_{\mathcal{A}'}\circ f$.
For any $i\in \{1,2,\cdots,n\}$, we have
\begin{align*}\label{Eq1}\tag{Eq1}
&\quad\quad\quad \partial_{\mathcal{A}'}\circ f(x_i)=\partial_{\mathcal{A}'}(\sum\limits_{j=1}^na_{ij}x_j')\\
&=\sum\limits_{j=1}^na_{ij}(\sum\limits_{l=1}^{j-1}t_l'x_l'x_j'+t_j'x_j'^2+\sum\limits_{l=j+1}^nt_l'x_j'x_l')\\
&=\sum\limits_{j=1}^na_{ij}x_j'(\sum\limits_{l=1}^{j-1}t_l'x_l'+t_j'x_j'+\sum\limits_{l=j+1}^nt_l'x_l')\\
&=(\sum\limits_{j=1}^na_{ij}x_j')(\sum\limits_{l=1}^nt_l'x_l')\\
&=(\sum\limits_{l=1}^na_{il}x_l')(\sum\limits_{s=1}^nt_s'x_s')
\end{align*}
and
\begin{align*}\label{Eq2}\tag{Eq2}
&\quad \quad\quad\quad f\circ \partial_{\mathcal{A}}(x_i)=f(\sum\limits_{j=1}^{i-1}t_jx_jx_i+t_ix_i^2+\sum\limits_{j=i+1}^nt_jx_ix_j)\\
&=\sum\limits_{j=1}^{i-1}t_j(\sum\limits_{l=1}^na_{jl}x_l')(\sum\limits_{l=1}^na_{il}x_l')+t_i(\sum\limits_{l=1}^na_{il}x_l')^2+\sum\limits_{j=i+1}^nt_j(\sum\limits_{l=1}^na_{il}x_l')(\sum\limits_{l=1}^na_{jl}x_l')\\
&=\sum\limits_{j=1}^{i-1}t_j(\sum\limits_{s=1}^na_{js}x_s')(\sum\limits_{l=1}^na_{il}x_l')+t_i(\sum\limits_{l=1}^na_{il}x_l')^2+\sum\limits_{j=i+1}^nt_j(\sum\limits_{l=1}^na_{il}x_l')(\sum\limits_{s=1}^na_{js}x_s')\\
&=(\sum\limits_{l=1}^na_{il}x_l')[\sum\limits_{j=1}^{i-1}t_j(\sum\limits_{s=1}^na_{js}x_s')+t_i(\sum\limits_{l=1}^na_{il}x_l')+\sum\limits_{j=i+1}^nt_j(\sum\limits_{s=1}^na_{js}x_s')]\\
&=(\sum\limits_{l=1}^na_{il}x_l')[\sum\limits_{j=1}^{i-1}t_j(\sum\limits_{s=1}^na_{js}x_s')+t_i(\sum\limits_{s=1}^na_{is}x_s')+\sum\limits_{j=i+1}^nt_j(\sum\limits_{s=1}^na_{js}x_s')]\\
&=(\sum\limits_{l=1}^na_{il}x_l')[\sum\limits_{j=1}^nt_j(\sum\limits_{s=1}^na_{js}x_s')]\\
&=(\sum\limits_{l=1}^na_{il}x_l')[\sum\limits_{s=1}^n(\sum\limits_{j=1}^na_{js}t_j)x_s'].
\end{align*}
Since $(a_{ij})_{n\times n} \in \mathrm{GL}_n(\k)$, we have $(a_{i1},a_{i2},\cdots,a_{in})\neq (0,0,\cdots, 0)$. So $\sum\limits_{l=1}^na_{il}x_l'\neq 0$ and $ \partial_{\mathcal{A}'}\circ f(x_i)=f\circ \partial_{\mathcal{A}}(x_i)$ implies that $t_s'=\sum\limits_{j=1}^na_{js}t_j$, for any $s\in \{1,2,\cdots,n\}$. Then we get
 $$(t_1',t_2',\cdots,t_n')=(t_1,t_2,\cdots,t_n)M.$$

 Conversely, if there exists a matrix $M=(a_{ij})_{n\times n}\in \mathrm{GL}_n(\k)$ such that $$(t_1',t_2',\cdots,t_n')=(t_1,t_2,\cdots,t_n)M,$$ then $t_s'=\sum\limits_{j=1}^na_{js}t_j$, for any $s\in \{1,2,\cdots,n\}$. Hence
 \begin{equation}\label{chainmap}(\sum\limits_{l=1}^na_{il}x_l')(\sum\limits_{s=1}^nt_s'x_s')=(\sum\limits_{l=1}^na_{il}x_l')[\sum\limits_{s=1}^n(\sum\limits_{j=1}^na_{js}t_j)x_s'], \forall i\in \{1,2,\cdots,n\}.
 \end{equation}
 Define a linear map $f: \mathcal{A}^1 \to \mathcal{A}'^1$
                        by
 $$\left(\begin{array}{c}
                           f(x_1) \\
                           f(x_2) \\
                           \vdots \\
                           f(x_n)
                         \end{array}
                       \right)= M\left(
                         \begin{array}{c}
                           x_1' \\
                           x_2' \\
                           \vdots \\
                           x_n'
                         \end{array}
                       \right).$$
 Obviously, $f$ is invertible since $M\in \mathrm{GL}_n(\k)$.  We have $$f(x_i)f(x_j)=f(x_j)f(x_i), \quad \forall 1\le i< j\le n,$$ since $\mathcal{A}'$ is a commutative algebra. Hence $f: \mathcal{A}^1\to \mathcal{A}'^1$ can be extended to a morphism of graded algebras between $\k[x_1,x_2,\cdots, x_n]$ and $\k[x_1',x_2',\cdots, x_n']$. We still denote it by $f$. For any $i\in \{1,2,\cdots,n\}$, we still have (\ref{Eq1}) and (\ref{Eq2}). Then (\ref{chainmap}) indicates that $f\circ \partial_{\mathcal{A}}(x_i)=\partial_{\mathcal{A}'}\circ f(x_i)$, which implies that $f$ is a chain map. Hence, $f:A\to A'$ is an isomorphism of DG algebras.
\end{proof}

\begin{cor}\label{isomor}
In space $\Omega(x_1,x_2,\cdots, x_n)$, there are only two isomorphism
 classes $\mathcal{A}(0,0,\cdots,0)$ and $\mathcal{A}(1,0,\cdots, 0)$.
\end{cor}
\begin{proof}
By Theorem \ref{isom}, $\mathcal{A}(t_1,t_2,\cdots,t_n)\cong \mathcal{A}(t_1',t_2',\cdots,t_n')$ if and only if there is $M\in \mathrm{GL}_n(\k)$ such that $(t_1',t_2',\cdots,t_n')=(t_1,t_2,\cdots,t_n)M$. For any $(t_1,t_2,\cdots,t_n)\neq (0,0,\cdots,0)$, there exists $i\in \{1,2,\cdots, n\}$ such that $t_i\neq 0$.
We have $$(t_1,t_2,\cdots, t_n)=(1,0,\cdots,0)\left(
                         \begin{array}{cccccccc}
                           t_1 & t_2& \cdots &t_{i-1} &t_i&t_{i+1} &\cdots & t_n \\
                           1 & 0  &\cdots  &0&0&0 &\cdots & 0 \\
                           0& 1&\cdots &0&0&0&\cdots & 0 \\
                           \vdots &\vdots &\ddots &\vdots &\vdots &\vdots &\ddots &\vdots \\
                           0&0&\cdots &1&0&0&\cdots & 0\\
                           0&0&\cdots &0&0&1&\cdots &0 \\
                           \vdots &\vdots &\ddots &\vdots &\vdots &\vdots &\ddots &\vdots \\
                           0&0&\cdots &0&0&0&\cdots &1 \\
                         \end{array}
                       \right).$$
Since $$\left|
                       \begin{array}{cccccccc}
                           t_1 & t_2& \cdots &t_{i-1} &t_i&t_{i+1} &\cdots & t_n \\
                           1 & 0  &\cdots  &0&0&0 &\cdots & 0 \\
                           0& 1&\cdots &0&0&0&\cdots & 0 \\
                           \vdots &\vdots &\ddots &\vdots &\vdots &\vdots &\ddots &\vdots \\
                           0&0&\cdots &1&0&0&\cdots & 0\\
                           0&0&\cdots &0&0&1&\cdots &0 \\
                           \vdots &\vdots &\ddots &\vdots &\vdots &\vdots &\ddots &\vdots \\
                           0&0&\cdots &0&0&0&\cdots &1 \\
                         \end{array}
                        \right| = (-1)^{i+1}t_i\neq 0,$$
we have $\mathcal{A}(t_1,t_2,\cdots,t_n)\cong \mathcal{A}(1,0,\cdots,0)$.
Clearly, $\mathcal{A}(1,0,\cdots,0)\neq \mathcal{A}(0,0,\cdots,0)$
since $(1,0,\cdots,0)\neq (0,0,\cdots,0)M$ for any $M\in \mathrm{GL}_n(\k)$. Therefore, there are only two isomorphism classes $\mathcal{A}(0,0,\cdots,0)$ and $\mathcal{A}(1,0,\cdots,0)$ in $\Omega(x_1,x_2,\cdots,x_n)$.
\end{proof}

\begin{cor}\label{auto} For any $\mathcal{A}(t_1,t_2,\cdots,t_n)\in \Omega(x_1,x_2,\cdots,x_n)$, we have
$$\mathrm{Aut}_{dg}(\mathcal{A}(t_1,t_2,\cdots,t_n))\cong \{M\in \mathrm{GL}_n(\k)| (t_1,t_2,\cdots,t_n)=(t_1,t_2,\cdots, t_n)M   \}.$$
\end{cor}

\begin{rem}
We want to consider the homological properties of DG polynomial algebras. This can be reduced to studying that of $\mathcal{A}(0,0,\cdots,0)$ and $\mathcal{A}(1,0,\cdots, 0)$, by Corollary \ref{isomor}. Corollary \ref{auto} characterizes the automorphism group of any DG polynomial algebra.
\end{rem}

\section{cohomology of polynomial dg algebras}
Generally, the cohomology algebra of a DG algebra usually contains much useful information on its properties (cf.\cite{AT,DGI}). In this section, we will compute the cohomology algebra of non-trivial DG polynomial algebra. By Corollary \ref{isomor}, we only need to compute $H(\mathcal{A}(1,0,\cdots,0))$. We have the following proposition.
\begin{prop}\label{cohomology}
 The cohomology algebra of $\mathcal{A}(1,0,\cdots,0)$ is the polynomial algebra  $$\k[\lceil x_2^2\rceil,\lceil x_2x_3\rceil, \cdots,\lceil x_2x_n\rceil, \lceil x_3^2\rceil, \cdots, \lceil x_3x_n\rceil, \cdots,  \lceil x_{n-1}^2\rceil, \lceil x_{n-1}x_n\rceil, \lceil x_n^2\rceil ].$$
\end{prop}
\begin{proof}
For simplicity, let $\mathcal{A}=\mathcal{A}(1,0,\cdots,0)$. We have $\partial_{\mathcal{A}}(x_i)=x_1x_i$, for any $i\in \{1,2,\cdots, n\}$. Hence, $$\mathrm{im}(\partial_{\mathcal{A}}^1)=\bigoplus_{i=1}^n \k x_1x_i.$$
For any $i,j\in \{1,2,\cdots,n\}$, we have
$$\partial_{\mathcal{A}}(x_ix_j)=\partial_{\mathcal{A}}(x_i)x_j-x_i\partial_{\mathcal{A}}(x_j) =x_1x_ix_j-x_ix_1x_j
                              =0.$$
Then $\mathrm{ker}(\partial_{\mathcal{A}}^2)=\mathcal{A}^2$ and hence $\mathrm{ker} (\partial_{\mathcal{A}}^{2k})=\mathcal{A}^{2k}$ by the Leibniz rule, for any $k\ge 2$.
Since $\mathcal{A}^2=\bigoplus_{i=1}^n\bigoplus_{j=i}^n \k x_ix_j$, we have $$H^2(\mathcal{A})=\bigoplus_{i=2}^n\bigoplus_{j=i}^n \k x_ix_j.$$
Since $\partial_{\mathcal{A}}(x_i)=x_1x_i$ and $\mathrm{ker}(\partial_{\mathcal{A}}^{2k-2})=\mathcal{A}^{2k-2}$, it is easy to check that
$$\mathrm{im}(\partial_{\mathcal{A}}^{2k-1})=\bigoplus_{\omega_1=1}^{2k}\bigoplus_{\stackrel{\sum\limits_{j=2}^n\omega_j=2k-\omega_1}{x_j\ge 0, j=2,\cdots,n}}\k x_1^{\omega_1}x_2^{\omega_2}\cdots x_n^{\omega_n}.$$ Since $$\mathcal{A}^{2k}=\bigoplus_{\stackrel{\sum\limits_{j=1}^n\omega_j=2k}{x_j\ge 0, j=1,\cdots,n}}\k x_1^{\omega_1}x_2^{\omega_2}\cdots x_n^{\omega_n},$$ we have $$H^{2k}(\mathcal{A})=\bigoplus_{\stackrel{\sum\limits_{j=2}^n\omega_j=2k}{x_j\ge 0, j=2,\cdots,n}}\k x_2^{\omega_2}\cdots x_n^{\omega_n}.$$
For any $k\ge 2$, any cocycle element in $\mathcal{A}^{2k+1}$ can be written as
$\sum\limits_{i=1}^nx_if_i$ for some $f_i\in \mathcal{A}^{2k}, i=1,2,\cdots, n$. We have $\partial_{\mathcal{A}}(\sum\limits_{i=1}^nx_if_i)=\sum\limits_{i=1}^nx_1x_if_i=x_1\sum\limits_{i=1}^nx_if_i=0$.
So $\sum\limits_{i=1}^nx_if_i=0$. Hence, $\mathrm{ker}(\partial_{\mathcal{A}}^{2k+1})=0$ and then $H^{2k+1}(\mathcal{A})=0$. Therefore,
$$H(\mathcal{A})=\k[\lceil x_2^2\rceil,\lceil x_2x_3\rceil, \cdots,\lceil x_2x_n\rceil, \lceil x_3^2\rceil, \cdots, \lceil x_3x_n\rceil, \cdots,  \lceil x_{n-1}^2\rceil, \lceil x_{n-1}x_n\rceil, \lceil x_n^2\rceil ].$$

\end{proof}

By Proposition \ref{cohomology}, one can deduce the following interesting results for DG polynomial algebras.

\begin{thm}\label{asreg}
Any DG polynomial algebra $\mathcal{A}(t_1,t_2,\cdots, t_n)$ is a formal, homologically smooth and Gorenstein DG algebra.
\end{thm}
\begin{proof}
For briefness, let $\mathcal{A}=\mathcal{A}(t_1,t_2,\cdots, t_n)$.

 If $(t_1,t_2,\cdots, t_n)=(0,0,\cdots, 0)$, then $\partial_{A}=0$ and $H(\mathcal{A})=\mathcal{A}^{\#}=k[x_1,x_2,\cdots, x_n]$.  Hence $H(\mathcal{A})$ is a Noetherian, Gorenstein graded algebra with $\mathrm{gl.dim}H(\mathcal{A})=n$. Therefore, $\mathcal{A}$ is Gorenstein and $\k\in \mathrm{D}^c(\mathcal{A})$ by \cite[Proposition 1]{Gam} and \cite[Corollary 3.6]{MW2}, respectively.  Hence $\mathcal{A}$ is homologically smooth by \cite[Corollary 2.7]{MW3}.

If $(t_1,t_2,\cdots, t_n)\neq (0,0,\cdots, 0)$, then $\mathcal{A}\cong \mathcal{A}(1,0,\cdots, 0)$ by Corollary \ref{isomor}. We have $$H(\mathcal{A})\cong\k[\lceil x_2^2\rceil,\lceil x_2x_3\rceil, \cdots,\lceil x_2x_n\rceil, \lceil x_3^2\rceil, \cdots, \lceil x_3x_n\rceil, \cdots,  \lceil x_{n-1}^2\rceil, \lceil x_{n-1}x_n\rceil, \lceil x_n^2\rceil ]$$
by Theorem \ref{cohomology}. So $H(\mathcal{A})$ is a Noetherian, Gorenstein graded algebra with $$\mathrm{gl.dim}H(\mathcal{A})=\frac{n(n-1)}{2}<\infty.$$ By \cite[Proposition 1]{Gam}, $\mathcal{A}$ is Gorenstein.  And $\mathcal{A}$ is homologically smooth by \cite[Corollary 3.6]{MW2} and \cite[Corollary 2.7]{MW3}. We can define a morphism of DG algebras $
\iota: (H(\mathcal{A}),0) \to (\mathcal{A},\partial_{\mathcal{A}})$ by
          $\iota(\lceil x_ix_j\rceil) = x_ix_j$, for all $2\le i\le j\le n$. One sees easily that $\iota$
is a quasi-isomorphism.  So $\mathcal{A}$ is formal.
\end{proof}

\section{Calabi-Yau properties of polynomial dg algebras}
By Theorem \ref{asreg}, we know that any DG polynomial algebra $\mathcal{A}(t_1,t_2,\cdots, t_n)$ is a formal, homologically smooth and Gorenstein DG algebra. It is natural for one to ask whether $\mathcal{A}(t_1,t_2,\cdots, t_n)$ is Calabi-Yau. We completely solve this problem in this section.
\begin{rem}
For any connected cochain DG algebra $\mathcal{A}$, one sees that $H(\mathcal{A}^e)=H(\mathcal{A})^e$ is a connected graded algebra.
We should emphasize that the multiplication of $H(\mathcal{A})^e$ is defined by $$(\lceil c\rceil\otimes \lceil d\rceil)\cdot (\lceil e\rceil\otimes \lceil f\rceil) =   (-1)^{|d|\cdot |e|} \lceil c\rceil \lceil e\rceil\otimes \lceil d\rceil\diamond \lceil f\rceil=(-1)^{|d|\cdot |e|+|d|\cdot |f|} \lceil ce\rceil\otimes \lceil fd\rceil,$$ for any cocycle elements $c, e\in \mathcal{A}$ and $d, f\in \mathcal{A}^{op}$. In the proof of the following two theorems, we need to construct minimal free resolutions of ${}_{H(\mathcal{A})^e}H(\mathcal{A})$ and $H(\mathcal{A})_{H(\mathcal{A})^e}$. We remind the readers to remember our emphasis above.
\end{rem}
\begin{thm}\label{cyprop}
Let $\mathcal{A}$ be a connected cochain DG algebra such that $$H(\mathcal{A})=\k[\lceil y_1\rceil, \cdots, \lceil y_m\rceil ],$$ for some central, cocycle and degree $2$ elements $y_1,\cdots, y_m$  in $\mathcal{A}$. Then $\mathcal{A}$ is a $(-m)$-Calabi-Yau DG algebra.
\end{thm}
\begin{proof}
The graded left $H(\mathcal{A})^e$-module $H(\mathcal{A})$ admits a minimal free resolution:
\begin{align*}
0\leftarrow H(\mathcal{A})\stackrel{\mu}{\leftarrow}H(\mathcal{A})^e\stackrel{d_1}{\leftarrow}F_1\stackrel{d_2}{\leftarrow}F_2\stackrel{d_3}{\leftarrow}\cdots\stackrel{d_{m-1}}{\leftarrow} F_{m-1}\stackrel{d_m}{\leftarrow} F_m  \stackrel{0}{\leftarrow} 0,
\end{align*}
where $\mu: H(\mathcal{A})^e\to H(\mathcal{A})$ is defined by $\mu(\lceil a\rceil \otimes \lceil b\rceil)=\lceil ab\rceil$, each $F_k$ is a free $H(\mathcal{A})^e$-module $$\bigoplus_{e_{i_1\cdots i_k}\in E_k} H(\mathcal{A})^ee_{i_1\cdots i_k}$$ of rank $C_m^k$ with basis $$E_k=\{e_{i_1\cdots i_k}|1\le i_1<i_2<\cdots <i_k\le m\},  \quad |e_{i_1\cdots i_k}|=2k,$$ and $d_k$ is defined by
\begin{align*}
d_k(e_{i_1\cdots i_k})=\sum\limits_{j=1}^k(-1)^{j-1} (\lceil y_{i_j}\rceil\otimes 1-1\otimes \lceil y_{i_j}\rceil)e_{i_1\cdots \hat{i_j}\cdots i_k},
\end{align*}
for any $k=1,2,\cdots, m$.
Applying the constructing procedure of Eilenberg-Moore resolution, we can construct a minimal semi-free resolution $F$ of the DG $\mathcal{A}^e$-module $\mathcal{A}$. Since $y_1,y_2,\cdots, y_m$ are central elements in $\mathcal{A}$, it is not difficult to check that
$$F^{\#}=\mathcal{A}^{e\#}\oplus [\bigoplus_{k=1}^m \bigoplus_{e_{i_1\cdots i_k}\in E_k}\mathcal{A}^{e\#}\Sigma^ke_{i_1\cdots i_k}]$$ and
$\partial_{F}$ is defined by
\begin{align*}
\partial_{F}(\Sigma^ke_{i_1\cdots i_k})=\sum\limits_{j=1}^k(-1)^{j-1} (y_{i_j}\otimes 1-1\otimes y_{i_j})\Sigma^{k-1}e_{i_1\cdots \hat{i_j}\cdots i_k},
\end{align*}
for any $e_{i_1\cdots i_k}\in E_k$ and $k=1,2,\cdots, m$. Clearly, $\mathcal{A}$ is homologically smooth since $F$ has a finite semi-basis.

Similarly, the right graded $H(\mathcal{A})^e$-module $ H(\mathcal{A})$ has a minimal free resolution
\begin{align}\label{right}
0\leftarrow  H(\mathcal{A})\stackrel{\tau}{\leftarrow} H(\mathcal{A})^e\stackrel{d_1}{\leftarrow}G_1\stackrel{d_2}{\leftarrow}G_2\stackrel{d_3}{\leftarrow}\cdots\stackrel{d_{m-1}}{\leftarrow} G_{m-1}\stackrel{d_m}{\leftarrow} G_m  \stackrel{0}{\leftarrow} 0,
\end{align}
where $\tau: H(\mathcal{A})^e\to H(\mathcal{A})$ is defined by $\tau(\lceil a\rceil \otimes \lceil b\rceil)=(-1)^{|a|\cdot|b|}\lceil ba\rceil$, each $G_k$ is a free right $H(\mathcal{A})^e$-module $$\bigoplus_{\lambda_{i_1\cdots i_k}\in \Lambda_k}\lambda_{i_1\cdots i_k} H(\mathcal{A})^e$$ of rank $C_m^k$ with basis $$\Lambda_k=\{\lambda_{i_1\cdots i_k}|1\le i_1<i_2<\cdots <i_k\le m\}, \quad |\lambda_{i_1\cdots i_k}|=2k,$$ and the differentials are defined by
\begin{align*}
&d_1(\lambda_{i_1})= (1\otimes \lceil y_{i_1}\rceil-\lceil y_{i_1}\rceil\otimes 1),\quad   \forall \lambda_{i_1}\in \Lambda_1                \\
&d_k(\lambda_{i_1\cdots i_k})=\sum\limits_{j=1}^k(-1)^{k-j}\lambda_{i_1\cdots \hat{i_j}\cdots i_k} (1\otimes \lceil y_{i_j}\rceil-\lceil y_{i_j}\rceil\otimes 1), \quad \forall \lambda_{i_1\cdots i_k} \in \Lambda_k,
\end{align*}
$k=2,\cdots, m.$ From the free resolution (\ref{right}), we can construct the Eilenberg-Moore resolution $G$ of $\mathcal{A}_{\mathcal{A}^e}$. We have
$$G^{\#}=\mathcal{A}^{e\#}\oplus [\bigoplus_{k=1}^m \bigoplus_{\lambda_{i_1\cdots i_k}\in \Lambda_k}\Sigma^k\lambda_{i_1\cdots i_k}\mathcal{A}^{e\#}]$$ and
$\partial_{G}$ is defined by
\begin{align*}
&\partial_{G}(1)=0, \quad\partial_{G}(\Sigma \lambda_{i_1})=(1\otimes y_{i_1}-y_{i_1}\otimes 1), \quad \forall \lambda_{i_1}\in \Lambda_1\\
&\partial_{G}(\Sigma^k\lambda_{i_1\cdots i_k})=\sum\limits_{j=1}^k(-1)^{k-j} \Sigma^{k-1}\lambda_{i_1\cdots \hat{i_j}\cdots i_k}(1\otimes y_{i_j}-y_{i_j}\otimes 1),
\end{align*}
for any $\lambda_{i_1\cdots i_k}\in \Lambda_k$ and $k=2,\cdots, m$. Hence $$\Sigma^m G^{\#}= \Sigma^m1\mathcal{A}^{e\#}\oplus [\bigoplus_{k=1}^m \bigoplus_{\lambda_{i_1\cdots i_k}\in \Lambda_k}\Sigma^{m+k}\lambda_{i_1\cdots i_k}\mathcal{A}^{e\#}]$$ with
\begin{align*}
&\partial_{\Sigma^mG}(\Sigma^m1)=0, \quad\partial_{\Sigma^mG}(\Sigma^{m+1} \lambda_{i_1})=(-1)^m\Sigma^m1(1\otimes y_{i_1}-y_{i_1}\otimes 1), \quad \forall \lambda_{i_1}\in \Lambda_1\\
&\partial_{\Sigma^m G}(\Sigma^{m+k}\lambda_{i_1\cdots i_k})=\sum\limits_{j=1}^k(-1)^{m+k-j} \Sigma^{m+k-1}\lambda_{i_1\cdots \hat{i_j}\cdots i_k}(1\otimes y_{i_j}-y_{i_j}\otimes 1),
\end{align*}

The DG right $\mathcal{A}^e$-module $\Hom_{\mathcal{A}^e}(F,\mathcal{A}^e)$ is a minimal semi-free DG module whose underlying graded module is
\begin{align*}
\{\k1^*\oplus [\bigoplus_{k=1}^m \bigoplus_{e_{i_1\cdots i_k}\in E_k}\k(\Sigma^ke_{i_1\cdots i_k})^*]\}\otimes (\mathcal{A}^e)^{\#}.
\end{align*}
The differential $\partial_{\Hom}$ of $\Hom_{\mathcal{A}^e}(F,\mathcal{A}^e)$ is defined by $$\partial_{\Hom}(f)=\partial_{\mathcal{A}^e} \circ f -(-1)^{|f|}f\circ \partial_F,$$
for any graded element $f\in \Hom_{\mathcal{A}^e}(F,\mathcal{A}^e)$. So we have  $\partial_{\Hom}[(\Sigma^me_{1\cdots m})^*]=0$,
\begin{small}
\begin{align*}
&\partial_{\Hom}[(\Sigma^ke_{i_1\cdots i_k})^*]=\sum\limits_{e_{i_1\cdots i_jli_{j+1}\cdots i_k}\in E_{k+1}}(-1)^{j-k}(\Sigma^{k+1}e_{i_1\cdots i_jli_{j+1}\cdots i_k})^*(1\otimes y_l-y_l\otimes 1),
\end{align*}
\end{small}
for any $k\in \{1,2,\cdots, m-1\}$, and
\begin{align*}
&\partial_{\Hom}[1^*]=\sum\limits_{i=1}^m(\Sigma e_i)^*(1\otimes y_i-y_i\otimes 1).
\end{align*}
Define an $(\mathcal{A}^e)^{op}$-linear map
$ \theta: \Hom_{\mathcal{A}^e}(F,\mathcal{A}^e)\to \Sigma^m G$
 by $$ \theta: \begin{cases}
(\Sigma^me_{1\cdots m})^* \,\, \mapsto (-1)^{\frac{m(m+1)}{2}} \Sigma^m 1,   \\
(\Sigma^k e_{i_1\cdots i_k})^* \,\, \mapsto (-1)^{\sum\limits_{j=1}^k i_j}\Sigma^{2m-k}\lambda_{\overline{i_1\cdots i_{k}}}, \quad \forall k\in \{1,\cdots, m-1\},\\
1^*  \mapsto \,\, \Sigma^{2m} \lambda_{1\cdots m},
\end{cases}
$$
where $\overline{i_1\cdots i_k}$ is the $m-k$ integers arranged from small to large obtained by deleting $i_1,i_2,\cdots, i_k$ from $\{1,2,\cdots, m\}$. We claim that $\theta$ is a chain map. Indeed, we have $$\theta\circ \partial_{\Hom} [(\Sigma^me_{1\cdots m})^*] =0=(-1)^{\frac{m(m+1)}{2}}\partial_{\Sigma^m G}[\Sigma^m 1]= \partial_{\Sigma^m G}\circ \theta [(\Sigma^me_{1\cdots m})^*]$$ and $\theta\circ \partial_{\Hom}[1^* ]=\partial_{\Sigma^m G}\circ \theta[1^*]$ since  \begin{align*}
\quad\quad \theta\circ \partial_{\Hom}[1^* ]&=\theta[\sum\limits_{i=1}^m(\Sigma e_i)^*(1\otimes y_i-y_i\otimes 1)]\\
&=\sum\limits_{i=1}^m(-1)^{i}\Sigma^{2m-1}\lambda_{1\cdots(i-1)(i+1)\cdots m}(1\otimes y_i-y_i\otimes 1)
\end{align*}
and
\begin{align*}
\partial_{\Sigma^m G}\circ \theta[1^* ]& =\partial_{\Sigma^m G}[\Sigma^{2m} \lambda_{1\cdots m}]\\
&=\sum\limits_{i=1}^{m}(-1)^{m+m-i}\Sigma^{2m-1}\lambda_{1\cdots (i-1)(i+1)\cdots m}(1\otimes y_i-y_i\otimes 1)\\
&=\sum\limits_{i=1}^{m}(-1)^{i}\Sigma^{2m-1}\lambda_{1\cdots (i-1)(i+1)\cdots  m}(1\otimes y_i-y_i\otimes 1).
\end{align*} Furthermore, for any $k\in \{1,2,\cdots m-1\}$, we have
$$\partial_{\Sigma^m G}\circ \theta [(\Sigma^ke_{i_1\cdots i_k})^*]=\theta\circ \partial_{\Hom} [(\Sigma^ke_{i_1\cdots i_k})^*]$$ since
\begin{align*}
&\quad\quad \theta\circ \partial_{\Hom} [(\Sigma^ke_{i_1\cdots i_k})^*]\\
&=\theta [\sum\limits_{e_{i_1\cdots i_jli_{j+1}\cdots i_k}\in E_{k+1}}(-1)^{j-k}(\Sigma^{k+1}e_{i_1\cdots i_jli_{j+1}\cdots i_k})^*(1\otimes y_l-y_l\otimes 1) ]\\
&=\sum\limits_{e_{i_1\cdots i_jli_{j+1}\cdots i_k}\in E_{k+1}}(-1)^{j-k+\sum\limits_{j=1}^k i_j+l}\Sigma^{2m-k-1}\lambda_{\overline{i_1\cdots i_jli_{j+1}\cdots i_k}}(1\otimes y_l-y_l\otimes 1)
\end{align*}
and
\begin{align*}
&\quad\quad \partial_{\Sigma^m G}\circ \theta [(\Sigma^ke_{i_1\cdots i_k})^*]\\
&=\partial_{\Sigma^m G}[(-1)^{\sum\limits_{j=1}^k i_j}\Sigma^{2m-k}\lambda_{\overline{i_1\cdots i_{k}}}]\\
&=\sum\limits_{e_{i_1\cdots i_jli_{j+1}\cdots i_k}\in E_{k+1}}(-1)^{\sum\limits_{j=1}^k i_j +2m-k-l+j}\Sigma^{2m-k-1}\lambda_{\overline{i_1\cdots i_jli_{j+1}\cdots i_k}}(1\otimes y_l-y_l\otimes 1).
\end{align*}
So $\theta$ is an isomorphism of DG $\mathcal{A}^{op}$-modules and
$\Hom_{\mathcal{A}^e}(F,\mathcal{A}^e)\cong \Sigma^m\mathcal{A}$ in $\mathrm{D}((\mathcal{A}^e)^{op})$.
Therefore, $\mathcal{A}$ is a $(-m)$-Calabi-Yau DG algebra.
\end{proof}

\begin{cor}\label{non-zero}
The DG polynomial algebra $\mathcal{A}(t_1,t_2,\cdots, t_n)$ is a $\frac{-n(n-1)}{2}$-Calabi-Yau DG algebra if $(t_1,t_2,\cdots, t_n)\neq (0,0,\cdots, 0)$.
\end{cor}

\begin{proof}
By Corollary \ref{isomor}, we have $\mathcal{A}(t_1,t_2,\cdots, t_n)\cong \mathcal{A}(1,0,\cdots, 0)$. So we only need to show $\mathcal{A}(1,0,\cdots, 0)$ is a $\frac{-n(n-1)}{2}$-Calabi-Yau DG algebra. By Theorem \ref{cohomology}, the cohomology graded algebra of $\mathcal{A}(1,0,\cdots, 0)$ is
$$\k[\lceil x_2^2\rceil,\lceil x_2x_3\rceil, \cdots,\lceil x_2x_n\rceil, \lceil x_3^2\rceil, \cdots, \lceil x_3x_n\rceil, \cdots,  \lceil x_{n-1}^2\rceil, \lceil x_{n-1}x_n\rceil, \lceil x_n^2\rceil ].$$ Theorem \ref{cyprop} indicates that $\mathcal{A}(1,0,\cdots, 0)$ is a $\frac{-n(n-1)}{2}$-Calabi-Yau DG algebra.
\end{proof}

\begin{thm}\label{degreeone}
Let $\mathcal{A}$ be a connected cochain DG algebra such that $$H(\mathcal{A})=\k[\lceil y_1\rceil, \cdots, \lceil y_m\rceil ],$$ for some central, cocycle and degree $1$ elements $y_1,\cdots, y_m$  in $\mathcal{A}$. Then $\mathcal{A}$ is a Koszul, homologically smooth and Gorenstein DG algebra. Moreover, $\mathcal{A}$ is $0$-Calabi-Yau if and only if $m$ is an odd integer.
\end{thm}
\begin{proof}
 By the assumption $H(\mathcal{A})=\k[\lceil y_1\rceil, \cdots, \lceil y_m\rceil ]$ is a Koszul, Gorenstein and Noetherian graded algebra with $\mathrm{gl.dim}H(A)=m<\infty$. Hence $\mathcal{A}$ is Koszul, Gorenstein and homologically smooth by \cite[Proposition2.3]{HW}, \cite[Proposition 1]{Gam} and \cite[Corollary 3.7]{MW2}, respectively.

 Now, lets consider the Calabi-Yau properties of $\mathcal{A}$.
The graded left $H(\mathcal{A})^e$-module $H(\mathcal{A})$ admits a minimal free resolution:
\begin{align*}
0\leftarrow H(\mathcal{A})\stackrel{\mu}{\leftarrow}H(\mathcal{A})^e\stackrel{d_1}{\leftarrow}F_1\stackrel{d_2}{\leftarrow}F_2\stackrel{d_3}{\leftarrow}\cdots\stackrel{d_{m-1}}{\leftarrow} F_{m-1}\stackrel{d_m}{\leftarrow} F_m  \stackrel{0}{\leftarrow} 0,
\end{align*}
where $\mu: H(\mathcal{A})^e\to H(\mathcal{A})$ is defined by $\mu(\lceil a\rceil \otimes \lceil b\rceil)=\lceil ab\rceil$, each $F_k$ is a free $H(\mathcal{A})^e$-module $$\bigoplus_{e_{i_1\cdots i_k}\in E_k} H(\mathcal{A})^ee_{i_1\cdots i_k}$$ of rank $C_m^k$ with basis $$E_k=\{e_{i_1\cdots i_k}|1\le i_1<i_2<\cdots <i_k\le m\}, \quad |e_{i_1\cdots i_k}|=k,$$ and $d_k$ is defined by
\begin{align*}
d_k(e_{i_1\cdots i_k})=\sum\limits_{j=1}^k(-1)^{j-1} (\lceil y_{i_j}\rceil\otimes 1+(-1)^k1\otimes \lceil y_{i_j}\rceil)e_{i_1\cdots \hat{i_j}\cdots i_k},
\end{align*}
for any $k=1,2,\cdots, m$.
Applying the constructing procedure of Eilenberg-Moore resolution, we can construct a minimal semi-free resolution $F$ of the DG $\mathcal{A}^e$-module $\mathcal{A}$. Since $y_1,\cdots, y_m$ are central elements in $\mathcal{A}$, it is not difficult to check that
$$F^{\#}=\mathcal{A}^{e\#}\oplus [\bigoplus_{k=1}^m \bigoplus_{e_{i_1\cdots i_k}\in E_k}\mathcal{A}^{e\#}\Sigma^ke_{i_1\cdots i_k}]$$ and
$\partial_{F}$ is defined by
\begin{align*}
\partial_{F}(\Sigma^ke_{i_1\cdots i_k})=\sum\limits_{j=1}^k(-1)^{j-1} (y_{i_j}\otimes 1+(-1)^k1\otimes y_{i_j})\Sigma^{k-1}e_{i_1\cdots \hat{i_j}\cdots i_k},
\end{align*}
for any $e_{i_1\cdots i_k}\in E_k$ and $k=1,2,\cdots, m$.

The DG right $\mathcal{A}^e$-module $\Hom_{\mathcal{A}^e}(F,\mathcal{A}^e)$ is a minimal semi-free DG module whose underlying graded module is
\begin{align*}
\{\k1^*\oplus [\bigoplus_{k=1}^m \bigoplus_{e_{i_1\cdots i_k}\in E_k}\k(\Sigma^ke_{i_1\cdots i_k})^*]\}\otimes (\mathcal{A}^e)^{\#}.
\end{align*}
The differential $\partial_{\Hom}$ of $\Hom_{\mathcal{A}^e}(F,\mathcal{A}^e)$ is defined by $$\partial_{\Hom}(f)=\partial_{\mathcal{A}^e} \circ f -(-1)^{|f|}f\circ \partial_F,$$
for any graded element $f\in \Hom_{\mathcal{A}^e}(F,\mathcal{A}^e)$. So we have  $\partial_{\Hom}[(\Sigma^me_{1\cdots m})^*]=0$,

\begin{align*}
&\quad\quad\quad \partial_{\Hom}[(\Sigma^ke_{i_1\cdots i_k})^*]\\
&=\sum\limits_{e_{i_1\cdots i_jli_{j+1}\cdots i_k}\in E_{k+1}}(-1)^{j+1}(\Sigma^{k+1}e_{i_1\cdots i_jli_{j+1}\cdots i_k})^*(y_l\otimes 1-(-1)^k 1\otimes y_l),
\end{align*}
for any $k\in \{1,2,\cdots, m-1\}$, and
\begin{align*}
&\partial_{\Hom}[1^*]=\sum\limits_{i=1}^m(\Sigma e_i)^*(1\otimes y_i-y_i\otimes 1).
\end{align*}
If $m=2t$ is an even integer, we claim that $\mathcal{A}$ is not a Calabi-Yau DG algebra. It suffices to
show that $\Hom_{\mathcal{A}^e}(F,\mathcal{A}^e)\not\cong  \mathcal{A}$ in $\mathrm{D}((\mathcal{A}^e)^{op})$.
  We prove this with a proof by contradiction. If $\Hom_{\mathcal{A}^e}(F,\mathcal{A}^e)\cong  \mathcal{A}$ in $\mathrm{D}((\mathcal{A}^e)^{op})$, then there exist a DG right $\mathcal{A}^e$-module $P$ and two quasi-isomorphisms of right DG $\mathcal{A}^e$-modules $\phi: P\to  \mathcal{A}$ and  $\psi: P\to \Hom_{\mathcal{A}^e}(F,\mathcal{A}^e)$. The DG right $\mathcal{A}^e$-module $P$ admits a minimal semi-free resolution $\varepsilon:G\to P$ since $H(P)\cong H(\mathcal{A})$ is bounded below.
Then $$\psi\circ \varepsilon: G\to \Hom_{\mathcal{A}^e}(F,\mathcal{A}^e)$$ is a quasi-isomorphism of right DG $\mathcal{A}^e$-modules.  Since both $G$ and $\Hom_{\mathcal{A}^e}(F,\mathcal{A}^e)$ are minimal, $\psi\circ \varepsilon$ is an isomorphism by \cite[\S 12, Corollary 1.3]{AFH}. Then the composition $$\Hom_{\mathcal{A}^e}(F,\mathcal{A}^e)\stackrel{(\psi\circ \varepsilon)^{-1}}{\longrightarrow} G\stackrel{\varepsilon}{\longrightarrow} P \stackrel{\phi}{\longrightarrow} \mathcal{A}$$
is a quasi-isomorphism of DG right $\mathcal{A}^e$-modules. Set $g=\phi\circ \varepsilon\circ (\psi\circ \varepsilon)^{-1}$. Since $H(g)$ is an isomorphism and $\lceil(\Sigma^me_{1\cdots m})^*\rceil \neq 0$, we have $g[(\Sigma^me_{1\cdots m})^*]= a_0\neq 0$, for some $a_0\in \mathcal{A}^0=k$. For any $j\in \{1,2,\cdots, m\}$, we have $|(\Sigma^{m-1} e_{1\cdots j-1 j+1\cdots m})^*|=0$.
So there exists $a\in \mathcal{A}^0$ such that $g[(\Sigma^{m-1} e_{1\cdots j-1 j+1\cdots m})^*]= a$. We have
\begin{align*}
& \partial_{\mathcal{A}}\circ   g[(\Sigma^{m-1} e_{1\cdots (j-1) (j+1)\cdots m})^*]= \partial_{\mathcal{A}}(a)=0,\\
& \quad\quad g\circ \partial_{\Hom}[(\Sigma^{m-1} e_{1\cdots (j-1)(j+1)\cdots m})^*]\\
&= g[(-1)^j(\Sigma^me_{1\cdots m})^*(y_j\otimes 1-(-1)^{m-1}1\otimes y_j)]\\
&= (-1)^ja_0 \cdot (y_j\otimes 1+1\otimes y_j)=(-1)^j2a_0y_j
\end{align*}
Therefore, $\partial_{\mathcal{A}}\circ g \neq g\circ \partial_{\Hom}$ since $\mathrm{char}\, \k=0$ and $a_0\neq 0$ in $\mathcal{A}^0$. This implies that $g$ is not a chain map. Then we get a contradiction since $g$ is a morphism of DG right $\mathcal{A}^e$-modules.

Now, it remains to consider the case that $m=2t-1$ is an odd integer.
 The right graded $H(\mathcal{A})^e$-module $ H(\mathcal{A})$ has a minimal free resolution
\begin{align}\label{right}
0\leftarrow  H(\mathcal{A})\stackrel{\tau}{\leftarrow} H(\mathcal{A})^e\stackrel{d_1}{\leftarrow}G_1\stackrel{d_2}{\leftarrow}G_2\stackrel{d_3}{\leftarrow}\cdots\stackrel{d_{m-1}}{\leftarrow} G_{m-1}\stackrel{d_m}{\leftarrow} G_m  \stackrel{0}{\leftarrow} 0,
\end{align}
where $\tau: H(\mathcal{A})^e\to H(\mathcal{A})$ is defined by $\tau(\lceil a\rceil \otimes \lceil b\rceil)=(-1)^{|a|\cdot|b|}\lceil ba\rceil$, each $G_k$ is a free right $H(\mathcal{A})^e$-module $$\bigoplus_{\lambda_{i_1\cdots i_k}\in \Lambda_k}\lambda_{i_1\cdots i_k} H(\mathcal{A})^e$$ of rank $C_m^k$ with basis $$\Lambda_k=\{\lambda_{i_1\cdots i_k}|1\le i_1<i_2<\cdots <i_k\le m\}, \quad |\lambda_{i_1\cdots i_k}|=k,$$ and the differentials are defined by
\begin{align*}
&d_1(\lambda_{i_1})= (1\otimes \lceil y_{i_1}\rceil-\lceil y_{i_1}\rceil\otimes 1),\quad   \forall \lambda_{i_1}\in \Lambda_1                \\
&d_k(\lambda_{i_1\cdots i_k})=\sum\limits_{j=1}^k(-1)^{k-j}\lambda_{i_1\cdots \hat{i_j}\cdots i_k} (1\otimes \lceil y_{i_j}\rceil+(-1)^k\lceil y_{i_j}\rceil\otimes 1), \quad \forall \lambda_{i_1\cdots i_k} \in \Lambda_k,
\end{align*}
$k=2,\cdots, m.$ From the free resolution (\ref{right}), we can construct the Eilenberg-Moore resolution $G$ of $\mathcal{A}_{\mathcal{A}^e}$. We have
$$G^{\#}=\mathcal{A}^{e\#}\oplus [\bigoplus_{k=1}^m \bigoplus_{\lambda_{i_1\cdots i_k}\in \Lambda_k}\Sigma^k\lambda_{i_1\cdots i_k}\mathcal{A}^{e\#}]$$ and
$\partial_{G}$ is defined by
\begin{align*}
&\partial_{G}(1)=0, \quad\partial_{G}(\Sigma \lambda_{i_1})=(1\otimes y_{i_1}-y_{i_1}\otimes 1), \quad \forall \lambda_{i_1}\in \Lambda_1\\
&\partial_{G}(\Sigma^k\lambda_{i_1\cdots i_k})=\sum\limits_{j=1}^k(-1)^{k-j} \Sigma^{k-1}\lambda_{i_1\cdots \hat{i_j}\cdots i_k}(1\otimes y_{i_j}+(-1)^ky_{i_j}\otimes 1),
\end{align*}
for any $\lambda_{i_1\cdots i_k}\in \Lambda_k$ and $k=2,\cdots, m$.

Define an $(\mathcal{A}^e)^{op}$-linear map
$ \theta: \Hom_{\mathcal{A}^e}(F,\mathcal{A}^e)\to  G$
 by $$ \theta: \begin{cases}
(\Sigma^me_{1\cdots m})^* \,\, \mapsto  (-1)^{\frac{m(m+1)}{2}+m}1,   \\
(\Sigma^k e_{i_1\cdots i_k})^* \,\, \mapsto (-1)^{\sum\limits_{j=1}^ki_j+k}\Sigma^{m-k}\lambda_{\overline{i_1\cdots i_{k}}}, \quad \forall k\in \{1,\cdots, m-1\},\\
1^*  \mapsto \,\, \Sigma^{m} \lambda_{1\cdots m},
\end{cases}
$$
where $\overline{i_1\cdots i_k}$ is the $m-k$ integers arranged from small to large obtained by deleting $i_1,i_2,\cdots, i_k$ from $\{1,2,\cdots, m\}$. We claim that $\theta$ is a chain map. Indeed, we have $$\theta\circ \partial_{\Hom} [(\Sigma^me_{1\cdots m})^*] =0=\partial_{G}[(-1)^{\frac{m(m+1)}{2}+m}1]= \partial_{G}\circ \theta [(\Sigma^me_{1\cdots m})^*]$$ and $\theta\circ \partial_{\Hom}[1^* ]=\partial_{\Sigma^m G}\circ \theta[1^*]$ since  \begin{align*}
\quad\quad \theta\circ \partial_{\Hom}[1^* ]&=\theta[\sum\limits_{i=1}^m(\Sigma e_i)^*(1\otimes y_i-y_i\otimes 1)]\\
&=\sum\limits_{i=1}^m(-1)^{i+1}\Sigma^{m-1}\lambda_{1\cdots(i-1)(i+1)\cdots m}(1\otimes y_i-y_i\otimes 1)
\end{align*}
and
\begin{align*}
\partial_{G}\circ \theta[1^*]& =\partial_{G}[\Sigma^{m} \lambda_{1\cdots m}]\\
&=\sum\limits_{i=1}^{m}(-1)^{m-i}\Sigma^{m-1}\lambda_{1\cdots (i-1)(i+1)\cdots  m}(1\otimes y_i+(-1)^my_i\otimes 1)\\
&=\sum\limits_{i=1}^{m}(-1)^{2t-1-i}\Sigma^{m-1}\lambda_{1\cdots (i-1)(i+1)\cdots m}(1\otimes y_i-y_i\otimes 1)\\
&=\sum\limits_{i=1}^{m}(-1)^{1+i}\Sigma^{m-1}\lambda_{1\cdots (i-1)(i+1)\cdots m}(1\otimes y_i-y_i\otimes 1).
\end{align*}

 Furthermore, for any $k\in \{1,2,\cdots m-1\}$, we have
$$\partial_{G}\circ \theta [(\Sigma^ke_{i_1\cdots i_k})^*]=\theta\circ \partial_{\Hom} [(\Sigma^ke_{i_1\cdots i_k})^*]$$ since
\begin{align*}
&\quad\quad \theta\circ \partial_{\Hom} [(\Sigma^ke_{i_1\cdots i_k})^*]\\
&=\theta [\sum\limits_{e_{i_1\cdots i_jli_{j+1}\cdots i_k}\in E_{k+1}}(-1)^{j+1}(\Sigma^{k+1}e_{i_1\cdots i_jli_{j+1}\cdots i_k})^*(y_l\otimes 1-(-1)^k 1\otimes y_l) ]\\
&=\sum\limits_{e_{i_1\cdots i_jli_{j+1}\cdots i_k}\in E_{k+1}}(-1)^{j+\sum\limits_{s=1}^ki_s+l+k}\Sigma^{m-k-1}\lambda_{\overline{i_1\cdots i_jli_{j+1}\cdots i_k}}(y_l\otimes 1-(-1)^k 1\otimes y_l)
\end{align*}
and
\begin{align*}
&\quad\quad \partial_{G}\circ \theta [(\Sigma^ke_{i_1\cdots i_k})^*]\\
&=\partial_{G}[(-1)^{\sum\limits_{s=1}^k i_s +k}\Sigma^{m-k}\lambda_{\overline{i_1\cdots i_{k}}}]\\
&=\sum\limits_{e_{i_1\cdots i_jli_{j+1}\cdots i_k}\in E_{k+1}}(-1)^{\sum\limits_{s=1}^k i_s +m-l+j }\Sigma^{m-k-1}\lambda_{\overline{i_1\cdots i_jli_{j+1}\cdots i_k}}(1\otimes y_l-(-1)^{-k}y_l\otimes 1)\\
&=\sum\limits_{e_{i_1\cdots i_jli_{j+1}\cdots i_k}\in E_{k+1}}(-1)^{\sum\limits_{s=1}^k i_s-l+j-k }\Sigma^{m-k-1}\lambda_{\overline{i_1\cdots i_jli_{j+1}\cdots i_k}}(y_l\otimes 1-(-1)^k1\otimes y_l).
\end{align*}
So $\theta$ is an isomorphism of DG $\mathcal{A}^{op}$-modules and
$\Hom_{\mathcal{A}^e}(F,\mathcal{A}^e)\cong \mathcal{A}$ in $\mathrm{D}((\mathcal{A}^e)^{op})$.
Therefore, $\mathcal{A}$ is a $0$-Calabi-Yau DG algebra.
\end{proof}

\begin{cor}\label{zerodiff}
The trivial DG polynomial algebra $\mathcal{A}(0,0,\cdots, 0)$ is a  Koszul, homologically smooth and Gorenstein DG algebra.  Moreover, it is $0$-Calabi-Yau if and only if $n$ is an odd integer.
\end{cor}

\begin{rem}{\rm
In Theorem \ref{cyprop} and Theorem \ref{degreeone}, we need that the cocycle elements $y_1,y_2,\cdots, y_m$ are central in $\mathcal{A}$. This simplifies the construction of Eilenberg-Moore resolution. Without this subtle condition, we are unable to get the corresponding result.
 Note that Theorem \ref{degreeone} is not a generalization of \cite[Proposition 2.4, Theorem 2.7]{MH} for this reason.}
\end{rem}

\section*{Acknowledgments}
 The first author is
supported by NSFC  (Grant No.11001056),
the
China Postdoctoral
Science Foundation  (Grant Nos.
20090450066 and 201003244),  the
Key Disciplines of Shanghai Municipality (Grant No.S30104) and the Innovation Program of
Shanghai Municipal Education Commission (Grant No.12YZ031).
\def\refname{References}

\end{document}